\documentclass[10pt]{amsart}
\usepackage{amsopn}
\usepackage{amssymb, amscd}

\newcommand{\nc}{\newcommand}

\nc{\vg}{\mathfrak{v} } \nc{\wg}{\mathfrak{w} }
\nc{\zg}{\mathfrak{z} } \nc{\ngo}{\mathfrak{n} }
\nc{\kg}{\mathfrak{k} } \nc{\mg}{\mathfrak{m} }
\nc{\bg}{\mathfrak{b} } \nc{\ggo}{\mathfrak{g} }
\nc{\ggob}{\overline{\mathfrak{g}} } \nc{\sog}{\mathfrak{so} }
\nc{\sug}{\mathfrak{su} } \nc{\spg}{\mathfrak{sp} }
\nc{\slg}{\mathfrak{sl} } \nc{\glg}{\mathfrak{gl} }
\nc{\cg}{\mathfrak{c} } \nc{\rg}{\mathfrak{r} }
\nc{\hg}{\mathfrak{h} } \nc{\tg}{\mathfrak{t} }
\nc{\ug}{\mathfrak{u} } \nc{\dg}{\mathfrak{d} }
\nc{\ag}{\mathfrak{a} } \nc{\pg}{\mathfrak{p} }
\nc{\sg}{\mathfrak{s} } \nc{\pca}{\mathcal{P}}
\nc{\nca}{\mathcal{N}} \nc{\lca}{\mathcal{L}}
\nc{\oca}{\mathcal{O}} \nc{\mca}{\mathcal{M}}
\nc{\tca}{\mathcal{T}} \nc{\aca}{\mathcal{A}}
\nc{\cca}{\mathcal{C}} \nc{\gca}{\mathcal{G}}
\nc{\sca}{\mathcal{S}} \nc{\hca}{\mathcal{H}}
\nc{\bca}{\mathcal{B}} \nc{\dca}{\mathcal{D}}
\nc{\val}{\operatorname{val}}

\nc{\vp}{\varphi} \nc{\ddt}{\tfrac{{\rm d}}{{\rm d}t}}
\nc{\im}{\mathtt{i}}

\renewcommand{\Im}{{\rm Im}}

\nc{\SO}{\mathrm{SO}} \nc{\Spe}{\mathrm{Sp}} \nc{\Sl}{\mathrm{SL}}
\nc{\SU}{\mathrm{SU}} \nc{\Or}{\mathrm{O}} \nc{\U}{\mathrm{U}}
\nc{\Gl}{\mathrm{GL}} \nc{\Se}{\mathrm{S}} \nc{\Cl}{\mathrm{Cl}}
\nc{\Spein}{\mathrm{Spin}} \nc{\Pin}{\mathrm{Pin}}
\nc{\G}{\mathrm{GL}_n}
\nc{\g}{\mathfrak{gl}_n}\nc{\Go}{\mathrm{GL}_8}

\nc{\RR}{{\Bbb R}} \nc{\HH}{{\Bbb H}} \nc{\CC}{{\Bbb C}}
\nc{\ZZ}{{\Bbb Z}} \nc{\FF}{{\Bbb F}} \nc{\NN}{{\Bbb N}}
\nc{\QQ}{{\Bbb Q}} \nc{\PP}{{\Bbb P}}

\nc{\vs}{\vspace{.2cm}} \nc{\vsp}{\vspace{1cm}}
\nc{\ip}{\langle\cdot,\cdot\rangle} \nc{\ipp}{(\cdot,\cdot)}
\nc{\la}{\langle} \nc{\ra}{\rangle} \nc{\unm}{\tfrac{1}{2}}
\nc{\unc}{\tfrac{1}{4}} \nc{\und}{\tfrac{1}{16}}
\nc{\no}{\vs\noindent} \nc{\lam}{\Lambda^2(\RR^n)^*\otimes\RR^n}
\nc{\tangz}{{\rm T}^{\rm Zar}} \nc{\nor}{{\sf n}}
\nc{\eigen}{(k_1<\cdots<k_r;d_1,\ldots,d_r)}
\nc{\eigencero}{(0<k_2<...<k_r;d_1,...,d_r)} \nc{\mum}{/\!\!/}
\nc{\kir}{/\!\!/\!\!/} \nc{\Ri}{\tfrac{4}{||\mu||^2}\Ric_{\mu}}
\nc{\ds}{\displaystyle} \nc{\ben}{\begin{enumerate}}
\nc{\een}{\end{enumerate}} \nc{\f}{\frac}

\nc{\crit}{\operatorname{Crit}}\nc{\Ricci}{\operatorname{Ric}}
\nc{\I}{\operatorname{I}}\nc{\di}{\operatorname{diag}}
\nc{\Id}{\operatorname{Id}} \nc{\He}{\operatorname{Hess}}
\nc{\ad}{\operatorname{ad}} \nc{\Ad}{\operatorname{Ad}}
\nc{\rgg}{\operatorname{rank}} \nc{\Irr}{\operatorname{Irr}}
\nc{\End}{\operatorname{End}} \nc{\Aut}{\operatorname{Aut}}
\nc{\Inn}{\operatorname{Inn}} \nc{\Der}{\operatorname{Der}}
\nc{\Ker}{\operatorname{Ker}} \nc{\Iso}{\operatorname{I}}
\nc{\Diff}{\operatorname{D}} \nc{\Lie}{\operatorname{L}}
\nc{\tr}{\operatorname{tr}} \nc{\dif}{\operatorname{d}}
\nc{\sen}{\operatorname{sen}} \nc{\modu}{\operatorname{mod}}
\nc{\Ric}{\operatorname{R}}
\nc{\Ricg}{\operatorname{Ric^{\gamma}}}
\nc{\Ricc}{\operatorname{Ric^{c}}} \nc{\sym}{\operatorname{sym}}
\nc{\symac}{\operatorname{sym^{ac}}}
\nc{\symc}{\operatorname{sym^{c}}} \nc{\scalar}{\operatorname{sc}}
\nc{\grad}{\operatorname{grad}} \nc{\ricci}{\operatorname{ric}}
\nc{\ricciac}{\operatorname{ric^{ac}}}
\nc{\riccic}{\operatorname{ric^{c}}}
\nc{\riccig}{\operatorname{ric^{\gamma}}}
\nc{\Rin}{\operatorname{M}} \nc{\Le}{\operatorname{L}}
\nc{\tang}{\operatorname{T}} \nc{\level}{\operatorname{level}}
\nc{\rad}{\operatorname{r}} \nc{\abel}{\operatorname{ab}}
\nc{\CH}{\operatorname{CH}} \nc{\mcc}{\operatorname{mcc}}
\nc{\Adj}{\operatorname{Adj}}

\theoremstyle{plain}
\newtheorem{theorem}{Theorem}[section]

\newtheorem{corollary}[theorem]{Corollary}
\newtheorem{lemma}[theorem]{Lemma}

\theoremstyle{definition}
\newtheorem{definition}[theorem]{Definition}

\theoremstyle{remark}

\title{Filiform nilsolitons of dimension $8$}

\author{Romina M. Arroyo}

\address{FaMAF and CIEM, Universidad Nacional de C\'ordoba, 5000 C\'ordoba, Argentina}
\email{arroyo@mate.uncor.edu}

\thanks{
Supported by a fellowship from CONICET and a research grants from
CONICET, FONCYT and Secyt (UNC)}

\begin{document}

\maketitle

\begin{abstract}
A Riemannian manifold $(M,g)$ is said to be {\it Einstein} if its
Ricci tensor satisfies $\ricci(g)=cg$, for some $c \in \RR$. In the
homogeneous case, a problem that is still open is the so called {\it
Alekseevskii Conjecture}. This conjecture says that any homogeneous
Einstein space with negative scalar curvature (i.e. $c<0$) is a {\it
solvmanifold}: a simply connected solvable Lie group endowed with a
left invariant Riemannian metric. The aim of this paper is to
classify Einstein solvmanifolds of dimension $9$ whose nilradicals
are {\it filiform} (i.e. $(n-1)$-step nilpotent and
$n$-dimensional).

\end{abstract}

\section{Introduction}\label{intro}

A Riemannian manifold $(M,g)$ is said to be {\it Einstein} if its
Ricci tensor satisfies $\ricci(g)=cg$, for some $c \in \RR$.
Einstein metrics are often considered as the nicest, or most
distinguished metrics on a given differentiable manifold (see for
instance \cite[Introduction]{Bss}).

In the homogeneous case, a problem that is still open is the so
called {\it Alekseevskii Conjecture} (see \cite[7.57]{Bss}). This
conjecture says that any homogeneous Einstein space with negative
scalar curvature (i.e. $c<0$) is a {\it solvmanifold}: a simply
connected solvable Lie group endowed with a left invariant
Riemannian metric. It is important to note that, nowadays, it is still unknown which solvable Lie groups admit a left invariant Einstein metric.

In \cite{standard}, Lauret has proved that any Einstein
solvmanifold $S$ is {\it standard} (i.e. $[\ag,\ag]=0$, where
$\ag:=[\sg,\sg]^{\perp}$, $\sg$ the Lie algebra of $S$), and the
study of standard Einstein solvmanifolds has been reduced to the
rank-one case, that is, $\dim \ag =1$, where strong structural and
uniqueness results are well known (see \cite{Hbr}).

A nilpotent Lie algebra $\ngo$ is called an {\it Einstein
nilradical} if it is the nilradical (i.e. maximal nilpotent ideal)
of the Lie algebra of an Einstein solvmanifold. It is proved in
\cite{soliton} that $\ngo$ is an Einstein nilradical if and only
if $\ngo$ admits an inner product $\ip$ such that
\begin{equation}\label{nilr}
\Ricci_{\ip}=cI+\phi, \qquad c \in \RR,\qquad \phi \in \Der(\ngo),
\end{equation}
where $\Ricci_{\ip}$ denotes the Ricci operator of the nilpotent Lie group $N$ endowed with the left invariant
metric determined by $\ip$. These metrics are called {\it nilsolitons}, as they are Ricci soliton metrics:
solutions of the Ricci flow which evolves only by scaling and the action of diffeomorphisms. Nilsolitons are
unique up to isometry and scaling, and $\sg$ is completely determined by the Lie algebra $\ngo=[\sg,\sg]$. So
the study of Einstein solvmanifolds is actually a problem on nilpotent Lie algebras (see the survey \cite{surv}
for further information).

In \cite{Wll}, it is proved that any nilpotent Lie algebra of dimension $\leq 6$ is an Einstein nilradical. A
first obstruction for a nilpotent Lie algebra to be an Einstein nilradical is that it has to admit an {\it
$\NN$-gradation} (i.e. $\ngo=\ngo_1\oplus \cdots \oplus\ngo_r$ such that $[\ngo_{i},\ngo_{j}] \subset
\ngo_{i+j}$) (see \cite{Hbr}). This condition is necessary but it is not sufficient, the first examples of $\NN$-graded Lie
algebra which are not Einstein nilradicals were found in \cite{einsteinsolv} and some of them are $7$-dimensional.

A nilpotent Lie algebra $\ngo$ is said to be {\it filiform} if $\dim{\ngo}=n$ and $\ngo$ is $(n-1)$-step
nilpotent. The classification of filiform Einstein nilradicals for $\dim{\ngo}\leq 7$ has been obtained in
\cite{einsteinsolv}.

The aim of this paper is to classify nilpotent filiform Lie algebras of dimension $8$ which are Einstein
nilradicals. After some preliminaries in Section \ref{pre}, we will begin our study in Section \ref{cla} by
classifying nilpotent filiform $\NN$-graded Lie algebras of dimension $8$ up to isomorphism. We will follow the
paper \cite{Nk13}, which is in turn based on results given in \cite{GzHkm}. Table \ref{ta} shows a complete
classification up to isomorphism of $\NN$-graded filiform Lie algebras of dimension $8$. In Section \ref{cle},
we use Table \ref{ta} and a characterization given in \cite{Nk13} to obtain the classification of filiform Lie
algebras of dimension $8$ which are Einstein nilradicals, which is the content of Table \ref{tad}.

\section{Preliminaries}\label{pre}

We consider the vector space
$$
V=\lam=\{\mu:\RR^n\times\RR^n\longrightarrow\RR^n : \mu\;
\mbox{bilinear and skew-symmetric}\},
$$
then $$ \nca=\{\mu\in V:\mu\;\mbox{satisfies Jacobi and is
nilpotent}\}
$$
is an algebraic subset of $V$ as the Jacobi identity and the
nilpotency condition can both be written as zeroes of polynomial
function. There is a natural action of $\G:=\Gl_n(\RR)$ on $V$
given by
\begin{equation*}\label{action}
g.\mu(X,Y)=g\mu(g^{-1}X,g^{-1}Y), \qquad X,Y\in\RR^n, \quad
g\in\G,\quad \mu\in V.
\end{equation*}

Note that each $\mu \in \nca$ defines a nilpotent Lie algebra
given by $(\RR^{n},\mu)$, thus $\nca$ parameterizes the set of all
nilpotent Lie algebras of dimension $n$. Note also that $\mu$ and
$\lambda$ are isomorphic if and only if they lie in the same
$\G$-orbit and so $\nca / \G$ parameterizes the set of isomorphism
classes.

In the filiform case most of the standard invariants from Lie theory (descending and ascending central series,
center, rank, etc) coincide, so next result will be a very useful tool to show that two filiform Lie algebras
are not isomorphic.

\begin{lemma}\label{le}
Let $\ngo$ be a Lie algebra and let $C_{j}(\ngo)$ be its descending
central series (i.e. $C_{0}(\ngo)=\ngo$ and $C_{i+1}(\ngo)=
[\ngo,C_{i}(\ngo)]$, $i \geq 0$). Consider
$$\I_{s,j}(\ngo)=\{[x] \in \ngo / C_{j}(\ngo): \dim \Im (\ad_{x})=s\},\quad s,j \in \NN.$$  If $\vp: \ngo
\longrightarrow \ngo'$ is an isomorphism of Lie algebras then $\tilde{\vp}(\I_{s,j}(\ngo))=\I_{s,j}(\ngo')$,
where $\tilde{\vp}:\ngo/C_{j}(\ngo) \longrightarrow \ngo'/C_{j}(\ngo')$ is the morphism of Lie algebras given by
$\tilde{\vp}[x]=[\vp(x)]'$, $x \in \ngo$.
\end{lemma}
\begin{proof}
Observe first that ${\tilde \vp}$ is an isomorphism of Lie algebras
since $C_{j}(\ngo)$ and $C_{j}(\ngo')$ are ideals of $\ngo$ and
$\ngo'$ respectively and $\vp$ is an isomorphism of Lie algebras.

Let $x \in \ngo$ and we define $V=\Im(\ad_{x})$ and $W=\Im
(\ad_{\vp(x)})$, so $W=\vp(V)$ because
$$
\begin{array}{rcl}
W & = & \{[\vp(x),z]:z\in\ngo'\}=\{[\vp(x),\vp(y)]: y\in
\ngo\}=\{\vp([x,y]):y \in \ngo\} \\
 & = & \vp(\{[x,y]: y \in \ngo\})=\vp(V).
\end{array}
$$
We define also $\overline{V}=\Im(\ad_{[x]})$ y
$\overline{W}=\Im(\ad_{[\vp(x)]'})$. Then
$\overline{V}=\{[[x],[y]]:y \in \ngo\}=\pi(V)$ y
$\overline{W}=\{[[\vp(x)]',[\vp(y)]']:y \in \ngo\}=\pi'(W)$, where
$\pi$ and $\pi'$ denote the canonical projections, so, as
$\tilde{\vp}$ is a isomorphism
$\overline{W}=\pi'(W)=\pi'(\vp(V))=\tilde{\vp}(\pi(V) )$ and then
$$
\begin{array}{rcl}
\Im(\ad_{[\vp(x)]'}) & = &
\pi'(\Im\ad_{\vp(x)})=\pi'(\vp(\Im\ad_{x}))=\tilde{\vp}(\pi(\Im\ad_{x})) \\
 & = & \tilde{\vp}(\Im\ad_{[x]}).
\end{array}
$$
Finally $\tilde{\vp}(\I_{s,j}(\ngo))=\I_{s,j}(\ngo')$.
\end{proof}

Let $G$ be a Lie group with Lie algebra $\ggo$. A left invariant
metric on $G$ will be always identified with the inner product
$\ip$ determined on $\ggo$. The pair $(\ggo,\ip)$ will be referred
to as a {\it metric Lie algebra}.

A consequence of the Alekseevskii Conjecture (see
\cite[7.57]{Bss}), is that a Lie group with a left invariant
metric which is Einstein should be either solvable or compact.
This conjecture is still open.

Any solvable Lie group admits at most one standard Einstein metric up to isometry and scaling. If $S$ is
standard Einstein, then for some distinguished element $H \in \ag$, the eigenvalues of $\ad{H}|_{\ngo}$ are all
positive integers without a common divisor, say $k_{1}< \cdots <k_{r}$. These results were proved by Jens Heber
in \cite{Hbr}.

\begin{definition}\label{eig}
Let $S$ be a standard Einstein solvmanifold and let
$d_{1},\ldots,d_{r}$ be the corresponding multiplicities of
$k_{1}< \cdots <k_{r}$, then the tuple
$$(k;d)=\eigen$$ is called the {\it eigenvalue type} of $S$.
\end{definition}

In every dimension, only finitely many eigenvalue types occur (see
\cite{Hbr}).

\begin{definition}\label{ext}
Let $(\ngo,\ip)$ be a metric nilpotent Lie algebra. A metric solvable Lie algebra $(\sg=\ag \oplus \ngo,\ip')$
such that $[\sg,\sg]=\ngo$ is called a {\it metric solvable extension} of $(\ngo,\ip)$ if $\ngo$ is an ideal of
$\sg$, $[\ag,\ag] \subset \ngo$ and $\ip' \mid_{\ngo \times \ngo}=\ip$.
\end{definition}

\begin{definition}\label{deriva}
A real semisimple derivation $\phi$ of a nilpotent Lie algebra
$\ngo$ is called {\it pre-Einstein} if $\tr(\phi \circ
\psi)=\tr\psi$ for all $\psi \in \Der(\ngo)$.
\end{definition}

A pre-Einstein derivation always exists, is unique up to conjugation and its eigenvalues are rational (see
\cite[Theorem 1]{Nkl2}). If $\ngo$ is an Einstein nilradical, then the derivation $\phi$ of (\ref{nilr}),
which is a multiple of $\ad(H)$, is a pre-Einstein derivation (up to scaling).

\begin{lemma}\label{de}
Let $(\ngo,\ip)$ be a metric nilpotent Lie algebra. Then there exists at most one pre-Einstein derivation $\phi$
of $\ngo$ symmetric with respect to $\ip$.
\end{lemma}
\begin{proof}
Let $\pg=\Der(\ngo) \cap \sym(\ngo,\ip)$ be. We define on $\pg$
the next inner product: $(A,B)=\tr(AB)$, $A,B \in \pg$.

As $f:\pg \longrightarrow \RR$, $f(A)=\tr(A)$ is a linear
functional then there exists an only $B \in \pg$ such that
$\tr(A)=\tr(AB)$ for all $A \in \pg$, and if $\phi \in \pg$ then
$\phi=B$.
\end{proof}

It follows from Lemma \ref{de} that a pre-Einstein derivation $\phi$
is symmetric with respect to an inner
product $\ip$, then $\ip$ is a nilsoliton if and only if $\ip$ satisfies condition (\ref{nilr}) for a multiple of $\phi$.\\

Let $\ngo=(\RR^{n}, \mu)$ be a nilpotent Lie algebra of dimension $n$ and let $\phi \in \Der(\ngo)$ a
pre-Einstein derivation of $\ngo$. Suppose that all the eigenvalues of $\phi$ are simple. Let $\{e_{i}\}$ be the
basis of eigenvectors for $\phi$ and
$$
\mu(e_{i},e_{j})=\sum_{k=1}^{n}c_{ij}^{k}e_{k},
$$ where $\mu$ is a Lie bracket of $\ngo$ (note that for every pair $(i,j)$, no more
than one of the $c_{ij}^{k}$ is nonzero). In the Euclidean space $\RR^{n}$ with the inner product
$(\cdot,\cdot)$ and orthonormal basis $\{f_{1},\ldots,f_{n}\}$ define the finite subset
$\mathbf{F}=\{\alpha_{ij}^{k}=f_{k}-f_{i}-f_{j}: c_{ij}^{k} \neq 0\}$. Let $L$ be the affine span of
$\mathbf{F}$, the smallest affine subspace of $\RR^{n}$ containing $\mathbf{F}$.

\begin{theorem}\label{nic}
\cite[Theorem 1]{Nk13} Let $\ngo$ be a nilpotent Lie algebra whose
pre-Einstein derivation has all the eigenvalues simple. The
algebra $\ngo$ is an Einstein nilradical if and only if the
orthogonal projection of the origin  of $\RR^{n}$ to $L$ lies in
the interior of the convex hull of $\mathbf{F}$.
\end{theorem}

\no If $N$=$\#\mathbf{F}$ and we introduce a matrix $Y \in \RR^{n
\times N}$ whose vector-columns are the vectors $\alpha_{ij}^{k}$ in
some fixed order. Define the vector $[1]_{N}=(1,\ldots,1)^{t} \in
\RR^{N}$ and the matrix $U \in \RR ^{N \times N}$ by $U=Y^{t}Y$. An
equivalent way to write Theorem \ref{nic} is the following

\begin{corollary}\label{ni}
\cite[Corollary 1]{Nk13} A Lie algebra $\ngo$ whose pre-Einstein derivation has
all the eigenvalues simple is an Einstein nilradical if and only
if there exists a vector $v \in \RR^{N}$ all of whose coordinates
are positive such that
\begin{equation}\label{u}
Uv=[1]_{N}.
\end{equation}
\end{corollary}

\section{Classification of filiform $\NN$-graded Lie algebras of dimension $8$}\label{cla}

In this section, we will study filiform Lie algebras of dimension
$8$ admitting  an $\NN$-{\it gradation}.

If $\ngo$ is a $n$-dimensional nilpotent filiform $\NN$-graded Lie
algebra then its {\it rank}, denoted by $\rgg\ngo$ (i.e. the
dimension of the maximal abelian subalgebra of $\Der(\ngo)$
consisting of real semisimple elements), is at most two. It is known from \cite[Section 3]{GzHkm} that if $\rgg\ngo=2$ then $\ngo$ is
isomorphic to
$$
\begin{array}{lll}
\mg_{0}(n):\qquad & \mu(e_1,e_i)=e_{i+1}, \quad i=2,\ldots,n-1,\\
\\
\mg_{1}(n),\; n \;\mbox{even} : \qquad & \mu(e_1,e_i)=e_{i+1}, \quad i=2,\ldots,n-1, \\
& \mu(e_i,e_{n-i+1})=(-1)^{i}e_{n}, \quad i=2,\ldots,n-1,
\end{array}
$$
(see also \cite{Ver}), and if $\rgg\ngo=1$ then $\ngo$ belongs to
one and only one of the following classes:

\begin{equation}\label{aerre}
\begin{array}{ll}
   A_r,\quad 2 \leq r \leq n-3: & \mu(e_{1},e_{i})=e_{i+1},\quad i=2,\ldots,n-1,\\
    & \mu(e_{i},e_{j})=c_{ij} e_{i+j+r-2},\quad i,j\geq 2, i+j+r-2\leq
    n,
\end{array}
\end{equation}

and

\begin{equation}\label{berre}
\begin{array}{ll}
    B_r,\quad 2 \leq r \leq n-4: & \mu(e_{1},e_{i})=e_{i+1},\quad i=2,\ldots,n-2, \\
    & \mu(e_{i},e_{j})=c_{ij} e_{i+j+r-2},\quad i,j\geq 2, i+j+r-2\leq
    n-1,\\
    & \mu(e_i,e_{n+1-i})=(-1)^{i+1}e_n,
\end{array}
\end{equation}
with at least one $c_{ij}\neq 0$.

The rank can not be zero because in that case it would not admit an
$\NN$-gradation.

For a given dimension, to classify Lie algebras in the classes
$A_{r}$ and $B_{r}$ is involved. The numbers $c_{ij}$'s must satisfy
the Jacobi condition, and even after finding those which do, it is
difficult to obtain a complete list of these algebras up to
isomorphism. The main difficult comes from the fact that most common
invariants from Lie theory usually coincide for all algebras in one
of these families, specially in a curve.

Observe that two algebras which lie in different classes can not be
isomorphic, because if they were so they would have the same
pre-Einstein derivation (recall that when the rank is $1$ the
semisimple derivation is unique up to conjugation and scaling) and
that is absurd because the eigenvalues of this derivation for
$A_{r}$ are proportional to $(1,r,r+1,\ldots,n+r-2)$ and for $B_{r}$
are proportional to $(1,r,r+1,\ldots,n+r-3,n+2r-3)$ (see
\cite{Nk13}).

In this section we will study Lie algebras of the classes $A_{r}$
and $B_{r}$ for $n=8$ and then will show in Table \ref{ta} a
complete list up to isomorphism.

\begin{lemma}\label{lpados}
For $t \in \RR$, let $(\RR^{8},\mu_{t})$ be the Lie algebra
defined by
$$
\begin{array}{lll}
\mu_{t}(e_{1},e_{j})=e_{j+1}, \quad j=2,\ldots,7, &&
\mu_{t}(e_{2},e_{3})=te_{5},\\
\mu_{t}(e_{2},e_{4})=te_{6}, && \mu_{t}(e_{2},e_{5})=(t-1)e_{7},\\
\mu_{t}(e_{2},e_{6})=(t-2)e_{8} &&  \mu_{t}(e_{3},e_{4})=e_{7}\\
\mu_{t}(e_{3},e_{5})=e_{8}. && \\
\end{array}
$$
Then $\ngo_{t}$ is isomorphic to $\ngo_{s}$ if and only if $t=s$.
\end{lemma}

\begin{proof}
If $\ngo_{t} \simeq \ngo_{s}$ then there exists $\vp: \ngo_{t}
\longrightarrow \ngo_{s}$ isomorphism of Lie algebras, so
$\vp<e_j,\ldots,e_8>=<e_j,\ldots,e_8>$, $j=3,\ldots,8$, since
$$\vp
(\mu_{t}(\ngo_{t},\mu_{t}(\ngo_{t},\ldots,\mu_{t}(\ngo_{t},\ngo_{t})\ldots)))
=
\mu_{s}(\ngo_{s},\mu_{s}(\ngo_{s},\ldots,\mu_{s}(\ngo_{s},\ngo_{s})\ldots)).$$
Therefore the matrix of $\vp$ respect to the basis
$\{e_{1},\ldots,e_{8}\}$ is given by
$$
\vp=\left[%
\begin{smallmatrix}
  a_{1} & b_{1} &  &  &  &  &  &  \\
  a_{2} & b_{2} &  &  &  &  &  &  \\
  a_{3} & b_{3} & c_{3} &  &  &  &  &  \\
  a_{4} & b_{4} & c_{4} & d_{4}&  &  &  &  \\
  a_{5} & b_{5} & c_{5} & d_{5} & l_{5} &  &  &  \\
  a_{6} & b_{6} & c_{6} & d_{6} & l_{6} & f_{6} &  &  \\
  a_{7} & b_{7} & c_{7} & d_{7} & l_{7} & f_{7} & g_{7} &  \\
  a_{8} & b_{8} & c_{8} & d_{8} & l_{8} & f_{8} & g_{8} & h_{8} \\
\end{smallmatrix}%
\right].
$$
It is easy to see that $b_{1}=0$ by using that $\rgg
(\ad_{\vp(e_2)})=\rgg (\ad_{e_2})$.

Since $\vp$ is a morphism of Lie algebras,
$\vp(\mu_{t}(e_{i},e_{j}))=\mu_{s}(\vp(e_{i}),\vp(e_{j}))$ for all
$i,j$ and therefore:

$$
\begin{array}{llll}
c_{3}=a_{1}b_{2}, & d_{4}=a_{1}c_{3}, &
l_{5}=a_{1}d_{4}, & f_{6}=a_{1}l_{5}, \\
g_{7}=a_{1}f_{6}, & h_{8}=a_{1}g_{7}, &
tl_{5}=sb_{2}c_{3}, & (t-1)g_{7}=(s-1)b_{2}l_{5}.\\
\end{array}
$$

\no Then we get that $ta_{1}^{2}=sb_{2}$ and
$(t-1)a_{1}^{2}=(s-1)b_{2}$ and so $ta_{1}^{2} - a_{1}^{2}=sb_{2}
- b_{2}$. Finally, $a_{1}^{2}=b_{2}$ and therefore $t=s$.
\end{proof}

\begin{lemma}\label{lados}
Let $\ngo$ be a Lie algebra of dimension $8$ of class $A_{2}$.
Then $\ngo$ is isomorphic to one and only one of the following
algebras:
$$
\begin{array}{lll}
\mg_{2}(8)\qquad & \mu(e_1,e_i)=e_{i+1}, \quad i=2,\ldots,7,\\
 & \mu(e_2,e_i)=e_{i+2}, \quad i=3,\ldots,6.\\ \\
\ggo_{\alpha}(8), \alpha \in \RR \qquad &  \mu(e_1,e_i)=e_{i+1}, \quad i=2,\ldots,7,\\
&\mu(e_2,e_3)=(2+ \alpha)e_{5}, \quad \mu(e_2,e_4)=(2+\alpha)e_{6},\\
& \mu(e_2,e_5)=(1+\alpha)e_{7}, \quad \mu(e_3,e_4)=e_{7},\\
& \mu(e_3,e_5)=e_{8}, \quad \mu(e_2,e_6)=\alpha e_{8}.
\end{array}
$$
\end{lemma}
\begin{proof}
From the Jacobi identity and (\ref{aerre}) follows that
$c_{23}=c_{24}$, $c_{25}=c_{24}-c_{34}$, $c_{26}=c_{25}-c_{35}$ y
$c_{34}=c_{35}$. Therefore, the Lie algebras of class $A_{2}$ are:
$$
\begin{array}{ll}
\mu_{a,b}(e_{1},e_{j})=e_{j+1} \qquad j=2,\ldots,7, & \mu_{a,b}(e_{2},e_{3})=ae_{5},\\
\mu_{a,b}(e_{2},e_{4})=ae_{6}, & \mu_{a,b}(e_{2},e_{5})=(a-b)e_{7},\\
\mu_{a,b}(e_{2},e_{6})=(a-2b)e_{8}, &
\mu_{a,b}(e_{3},e_{4})=be_{7}, \\
\mu_{a,b}(e_{3},e_{5})=be_{8}, & \\
\end{array}
$$ with $a,b \in \RR$, some non zero.

If $b=0$, we get $\mu_{a,0} \simeq \mu_{1,0}$ since
$\mu_{1,0}=g_{a}.\mu_{a,0}$ by taking $g_{a} \in \Go$,
$g_{a}=\di(1,a,a,a,a,a,a,a)$, where $\di(a,b,c,d,e,f,g,h)$ denotes
the diagonal matrix with entries $a,b,c,\ldots,h$. If $b \neq 0$
then $\mu_{a,b} \simeq \mu_{\frac ab ,1}$ since $\mu_{\frac
ab,1}=g_{b}.\mu_{a,b}$ for $g_{b} \in \Go$,
$g_{b}=\di(1,b,\ldots,b)$. Let $t=\frac ab$ be, then
$\mu_{t,1}=\mu_{t}$, where $\mu_{t}$ is the bracket of Lemma
\ref{lpados} and so we know that $\mu_{t} \simeq \mu_{s}$ if and
only if $t=s$.

Let us now see that there does not exist $t \in \RR$ such that
$\mu_{t,1}$ is isomorphic to $\mu_{1,0}$.

Let $x=c_{1}e_{1}+c_{2}e_{2}+c_{3}e_{3}$ be. The matrices of
$\ad_{x}$ relative to $\mu_{t,1}$ and $\mu_{1,0}$ in terms of the
basis $\{e_{1},\ldots,e_{8}\}$ are respectively given by
$$
\ad_{x}=\left[%
\begin{smallmatrix}
  0 & 0 & 0 & 0 & 0 & 0 & 0 & 0 \\
  0 & 0 & 0 & 0 & 0 & 0 & 0 & 0 \\
  -c_{2} & c_{1} & 0 & 0 & 0 & 0 & 0 & 0 \\
  -c_{3} & 0 & c_{1} & 0 & 0 & 0 & 0 & 0 \\
  0 & -tc_{3} & tc_{2} & c_{1} & 0 & 0 & 0 & 0 \\
  0 & 0 & 0 & tc_{2} & c_{1} & 0 & 0 & 0 \\
  0 & 0 & 0 & c_{3} & (t-1)c_{2} & c_{1} & 0 & 0 \\
  0 & 0 & 0 & 0 & c_{3} & (t-2)c_{2} & c_{1} & 0 \\
\end{smallmatrix}%
\right], \qquad
\ad_{x}=\left[%
\begin{smallmatrix}
  0 & 0 & 0 & 0 & 0 & 0 & 0 & 0 \\
  0 & 0 & 0 & 0 & 0 & 0 & 0 & 0 \\
  -c_{2} & c_{1} & 0 & 0 & 0 & 0 & 0 & 0 \\
  -c_{3} & 0 & c_{1} & 0 & 0 & 0 & 0 & 0 \\
  0 & -c_{3} & c_{2} & c_{1} & 0 & 0 & 0 & 0 \\
  0 & 0 & 0 & c_{2} & c_{1} & 0 & 0 & 0 \\
  0 & 0 & 0 & 0 & c_{2} & c_{1} & 0 & 0 \\
  0 & 0 & 0 & 0 & 0 & c_{2} & c_{1} & 0 \\
\end{smallmatrix}%
\right].
$$

Then, by Lemma \ref{le}, $\mu_{t,1}$ it is not isomorphic to
$\mu_{1,0}$ because the first matrix never has rank $2$ while the
second matrix does ($c_{1}=0, c_{2}=0, c_{3} \neq 0$).
\end{proof}

We will omit the proof of the following two lemmas because they are similar to those for Lemmas \ref{lpados} and
\ref{lados}.

\begin{lemma}\label{lpatres}
For $t \in \RR$, let $(\RR^{8}, \mu_{t})$ be the Lie algebra
defined by:
$$
\begin{array}{lll}
\mu_{t}(e_{1},e_{j})=e_{j+1}, \quad j=2,\ldots,7, &&
\mu_{t}(e_{2},e_{3})=(t+1)e_{6},\\
\mu_{t}(e_{2},e_{4})=(t+1)e_{7}, && \mu_{t}(e_{2},e_{5})=te_{8},\\
\mu_{t}(e_{3},e_{4})=e_{8}. && \\
\end{array}
$$ Then $\ngo_{t}$ is isomorphic to $\ngo_{s}$ if and only if $t=s$.
\end{lemma}

\begin{lemma}\label{latres}
If $\ngo$ is a Lie algebra of class $A_{3}$ and dimension $8$ then
$\ngo$ is isomorphic to one and only one from the next algebras:
$$
\begin{array}{lll}
\ag_{t}(8), t \in \RR & \mu(e_1,e_i)=e_{i+1}, \quad i=2,\ldots,7, & \mu(e_{2},e_{3})=(t+1)e_{6},\\
                      & \mu(e_{2},e_{4})=(t+1)e_{7}, & \mu(e_{2},e_{5})=te_{8}, \\
                      & \mu(e_{3},e_{4})=e_{8}. & \\ \\

\cg_{1,0}(8) &  \mu(e_1,e_i)=e_{i+1}, \quad i=2,\ldots,7, & \mu(e_2,e_3)=e_{6},\\
             & \mu(e_2,e_4)=e_{7}, &  \mu(e_2,e_5)=e_{8}.\\
\end{array}
$$
\end{lemma}

\begin{lemma}\label{lacua}
Up to isomorphism there is only one Lie algebra of class $A_{4}$
and dimension $8$, defined by:
$$
\begin{array}{ll}
\dg_{1}(8) & \mu(e_1,e_i)=e_{i+1}, \quad i=2,\ldots,7,
                  \quad \mu(e_{2},e_{3})=e_{7}, \quad
                  \mu(e_{2},e_{4})=e_{8}.
\end{array}
$$
\end{lemma}
\begin{proof}
From the Jacobi identity and (\ref{aerre}) follows that
$c_{23}=c_{24}$, and therefore:
$$
\begin{array}{lll}
\mu_{a}(e_{1},e_{j})=e_{j+1} \qquad j=2,\ldots,7, &
\mu_{a}(e_{2},e_{3})=a e_{7} & \mu_{a}(e_{2},e_{4})=a e_{8}.
\end{array}
$$ with $a \neq 0$.
Since $a \neq 0$, $\mu_{a} \simeq \mu_{1}$ for $g_{a}=\di(1,\frac
1a,\ldots,\frac 1a)$ because $\mu_{a}=g_{a}.\mu_{1}$.
\end{proof}

\begin{lemma}\label{lacin}
If $\ngo$ is a Lie algebra of class $A_{5}$ and dimension $8$, then
$\ngo$ is isomorphic to
$$
\begin{array}{lll}
\hg_{1}(8) & \mu(e_1,e_i)=e_{i+1}, \quad i=2,\ldots,7, &
\mu(e_{2},e_{3})=e_{8}.
\end{array}
$$
\end{lemma}
\begin{proof}
From (\ref{aerre}) we get that the algebras of class $A_{5}$ and
rank $1$ are:
$$
\begin{array}{lll}
\mu_{a}(e_{1},e_{j})=e_{j+1} \qquad j=2,\ldots,7 &
\mu_{a}(e_{2},e_{3})=a e_{8},&  a \neq 0.
\end{array}
$$

Since $a \neq 0$ then $\mu_{a} \simeq \mu_{1}$ by taking
$g_{a}=\di(1,\frac 1a,\ldots,\frac 1a).$
\end{proof}

\begin{lemma}\label{lbdos}
Let $\ngo$ be a Lie algebra of class $B_{2}$ and dimension $8$,
then $\ngo$ is isomorphic to
$$
\begin{array}{lll}
\bg(8) &  \mu(e_{1},e_{i})=e_{i+1}, \quad i=2,\ldots,6, &
\mu(e_{2},e_{3})=-\unm e_{5}, \\
       &  \mu(e_{2},e_{4})=-\unm e_{6}, & \mu(e_{2},e_{5})=-\frac 32e_{7} \\
       &  \mu(e_{3},e_{4})=e_{7}, &  \mu(e_{i},e_{9-i})=(-1)^{i}e_{8}, \quad i=2,3,4.\\
\end{array}
$$
\end{lemma}
\begin{proof}
From (\ref{berre}) and the Jacobi identity follows:
$$
\begin{array}{lll}
\mu_{a}(e_{1},e_{j})=e_{j+1} \qquad j=2,\ldots,6, & \mu_{a}(e_{2},e_{3})=a e_{5} \\
\mu_{a}(e_{2},e_{4})=a e_{6} &  \mu_{a}(e_{2},e_{5})=3a e_{7}, \\
\mu_{a}(e_i,e_{9-i})=(-1)^{i+1}e_8, \quad i=2,3,4, &
\mu_{a}(e_{3},e_{4})=-2ae_{7},
\end{array}
$$ with $a \neq 0$.

Like $a \neq 0$, then $\mu_a \simeq \mu_{-\frac 12}$ since by taking
$g_a=\di(1,a,\ldots,a,a^{2})$ we get that $g_a.\mu_a=\mu_{1}$.
\end{proof}

\begin{lemma}\label{lbtres}
Let $\ngo$ be a Lie algebra of class $B_{3}$ of dimension $8$. Then
$\ngo$ is isomorphic to
$$
\begin{array}{lll}
\kg_{1}(8) & \mu(e_1,e_i)=e_{i+1}, \quad i=2,\ldots,6, & \mu(e_2,e_3)=e_{6},\\
 & \mu(e_{2},e_{4})=e_{7}, & \mu(e_i,e_{9-i})=(-1)^{i+1}e_8, \quad i=2,3,4,\\
\end{array}
$$
\end{lemma}
\begin{proof}
From (\ref{berre}) and the Jacobi identity we obtain:
$$
\begin{array}{lll}
\mu_{a}(e_{1},e_{j})=e_{j+1}, \qquad j=2,\ldots,6, & \mu_{a}(e_{2},e_{3})=a e_{6} \\
\mu_{a}(e_{2},e_{4})=a e_{7}, & \mu_{a}(e_i,e_{9-i})=(-1)^{i+1}e_8,
\quad i=2,3,4, \\
\end{array}
$$ with $a \neq 0.$

Then $g_a.\mu_a=\mu_1$ for $g_a=\di(1,a,a,\ldots,a,a^{2}).$
\end{proof}

\begin{lemma}\label{lbcua}
Up to isomorphism  there is only one Lie algebra of class $B_{4}$
of dimension $8$, defined by:
$$
\begin{array}{lll}
\sg_1(8) & \mu(e_{1},e_{j})=e_{j+1} \qquad j=2,\ldots,6,  &
\mu(e_{2},e_{3})=e_{7},\\
& \mu(e_i,e_{9-i})=(-1)^{i+1}e_8,\quad i=2,3,4.
\end{array}
$$

\end{lemma}
\begin{proof}
The algebras of (\ref{berre}) for $r=4$ are:
$$
\begin{array}{lll}
\mu_{a}(e_{1},e_{j})=e_{j+1} \qquad j=2,\ldots,6,  &
\mu_{a}(e_{2},e_{3})=a e_{7},\\
\mu_a(e_i,e_{9-i})=(-1)^{i+1}e_8, \qquad i=2,3,4,
\end{array}
$$with $a \neq 0$.

Let $g_a=\di(1,a,\ldots,a,a^{2})$ be, then $\mu_a \simeq \mu_1$,
since $g_a.\mu_a=\mu_1.$
\end{proof}

The results obtained in this section can be summarized in the
following theorem.

\begin{theorem}\label{alg}
Let $\ngo$ be  an $\NN$-graded, filiform Lie algebra of dimension
$8$. Then $\ngo$ is isomorphic to one and only one of the Lie
algebras in Table \ref{ta}.
\end{theorem}

\begin{table}
{\center
\begin{tabular}{|c|c|}
  \hline

 $\mg_{0}(8)$ & $\begin{array}{c}
                \mu(e_1,e_i)=e_{i+1}, \quad i=2,\ldots,7, \\
                \end{array}$\\
\hline $\mg_{1}(8)$ & $\begin{array}{c}
               \mu(e_1,e_i)=e_{i+1}, \quad i=2,\ldots,7, \\
               \mu(e_i,e_{9-i})=(-1)^{i}e_{8}, \quad
               i=2,\ldots,7.\\
               \end{array}$\\
\hline $\mg_{2}(8)$ & $\begin{array}{c}
                   \mu(e_1,e_i)=e_{i+1}, \quad i=2,\ldots,7, \\
                    \mu(e_2,e_i)=e_{i+2}, \quad i=3,\ldots,6.\\
                 \end{array}$\\

\hline $\ggo_{\alpha}(8), \alpha \in \RR$ & $\begin{array}{c}
                       \mu(e_1,e_i)=e_{i+1}, \quad i=2,\ldots,7,\\
                       \mu(e_2,e_3)=(2+ \alpha)e_{5}, \\
                       \mu(e_2,e_4)=(2+ \alpha)e_{6},\\
                       \mu(e_2,e_5)=(1+ \alpha)e_{7},\\
                       \mu(e_2,e_6)=\alpha e_{8},\\
                       \mu(e_3,e_4)=e_{7},\\
                       \mu(e_3,e_5)=e_{8}.
                     \end{array}$ \\
\hline $\ag_{t}(8), t \in \RR$ & $\begin{array}{c}
                             \mu(e_1,e_i)=e_{i+1}, \quad i=2,\ldots,7,\\
                             \mu(e_{2},e_{3})=(t+1)e_{6}, \quad \mu(e_{2},e_{4})=(t+1)e_{7},\\
                             \mu(e_{2},e_{5})=te_{8}, \quad
                             \mu(e_{3},e_{4})=e_{8},
                             \end{array}$ \\
\hline $\cg_{1,0}(8)$ & $\begin{array}{c}
                       \mu(e_1,e_i)=e_{i+1}, \quad i=2,\ldots,7,\\
                       \mu(e_2,e_3)=e_{6}, \quad \mu(e_2,e_4)=e_{7},\\
                       \mu(e_2,e_5)=e_{8},
                     \end{array}$ \\
\hline $\dg_{1}(8)$ & $\begin{array}{c}
                      \mu(e_1,e_i)=e_{i+1}, \quad i=2,\ldots,7,\\
                      \mu(e_{2},e_{3})=e_{7}, \quad
                      \mu(e_{2},e_{4})=e_{8},\\
                      \end{array}$\\
\hline $\hg_{1}(8)$ & $\begin{array}{c}
                       \mu(e_1,e_i)=e_{i+1}, \quad i=2,\ldots,7,\\
                       \mu(e_{2},e_{3})=e_{8},\\
                       \end{array}$\\
\hline $\bg(8)$ & $\begin{array}{c}
                   \mu(e_{1},e_{i})=e_{i+1}, \quad i=2,\ldots,6,\\
                    \mu(e_{2},e_{3})=-\unm e_{5}, \quad
                    \mu(e_{2},e_{4})=-\unm e_{6},\\
                    \mu(e_{2},e_{5})=-\frac 32e_{7} \quad
                    \mu(e_{3},e_{4})=e_{7}, \\
                    \mu(e_{i},e_{9-i})=(-1)^{i+1}e_{8}, \quad
                    i=2,3,4.
                 \end{array}$ \\
\hline $\kg_{1}(8)$ & $\begin{array}{c}
                   \mu(e_1,e_i)=e_{i+1}, \quad i=2,\ldots,6, \\
                   \mu(e_2,e_3)=e_{6}, \quad \mu(e_{2},e_{4})=e_{7}, \\
                   \mu(e_i,e_{9-i})=(-1)^{i+1}e_8,
                 \end{array}$\\

\hline $\sg_1(8)$ & $\begin{array}{c}
                       \mu(e_1,e_i)=e_{i+1}, \quad i=2,\ldots,6, \\
                       \mu(e_2,e_3)=e_{7},\\
                       \mu(e_i,e_{9-i})=(-1)^{i+1}e_8,\quad
                       i=2,3,4,\\
                      \end{array}$\\
\hline
\end{tabular}}
\vs \caption{$\NN$-graded filiform Lie algebras of dimension $8$ up
to isomorphism.}\label{ta}
\end{table}

\section{Classification of filiform Einstein nilradicals of dimension
$8$}\label {cle}

In this section we will determine which filiform Lie algebras of
dimension $8$ are Einstein nilradicals. To do so we will apply
Corollary \ref{ni}.

Recall that if $\rgg \ngo=0$ then $\ngo$ does not admit an
$\NN$-gradation so $\ngo$ is not Einstein nilradical. If $\rgg
\ngo=2$, then $\ngo$ is Einstein nilradical (see \cite[Theorem
4.2]{finding} for $\mg_{0}(8)$ and \cite[Theorem 35]{Pyn} for
$\mg_{1}(8)$). Therefore, to obtain a complete classification of
filiform Einstein nilradicals of dimension $8$ we must study only
the algebras of Table \ref{ta} which has rank $1$.

Every algebra $\ngo$ from $A_{r}$ or $B_{r}$ has only one semisimple
derivation, which is automatically a pre-Einstein derivation (up to
conjugation and scaling). As all eigenvalues of such as a derivation
are simple (they are proportional to $(1,r,r+1,\ldots,n+r-2)$ for
$A_{r}$ and to $(1,r,r+1,\ldots,n+r-3,n+2r-3)$ for $B_{r}$), the
question of whether or not $\ngo$ is an Einstein nilradical is
answered by Corollary \ref{ni}.

In what follows, we will analyze three cases. The first case will be
a Lie algebra which is not an Einstein nilradical, the second one an
algebra which fails to be, and the third case will be a curve of Lie
algebras that depend on a real parameter. This curve has only one
Lie algebra which is not an Einstein nilradical. The other cases are
analogous to the three mentioned above.

We considerer in $\RR^{n}$ the canonical inner product
$(\cdot,\cdot)$ and $\{f_{1},\ldots,f_{n}\}$ the canonical basis.\\

\textbf{Case 1.} In this case, we consider the Lie algebra $\cg_{1,0}(8)$ of dimension $8$. $\cg_{1,0}(8)$ is
nilpotent and has a pre-Einstein derivation with simple eigenvalues, $\phi(e_{1})=e_{1}$,
$\phi(e_{i})=(i+1)e_{i} \quad i \geq 2$, with eigenvalues $\{1,3,4,5,6,7,8,9\}$.

The finite set $\mathbf{F}=\{\alpha_{ij}^{k}:c_{ij}^{k} \neq 0\}$
is given by:
$$
\begin{array}{ccc}
\{(-1,-1,1,0,0,0,0,0), & (-1,0,-1,1,0,0,0,0), & (-1,0,0,-1,1,0,0,0), \\
\ (-1,0,0,0,-1,1,0,0), & (-1,0,0,0,0,-1,1,0), & (-1,0,0,0,0,0,-1,1), \\
\ (0,-1,-1,0,0,1,0,0), & (0,-1,0,-1,0,0,1,0), & (0,-1,0,0,-1,0,0,1)\}, \\
\end{array}
$$ \no and with respect to this enumeration the matrix $U$ defined in
Section \ref{pre} is

$$
U=\left[%
\begin{smallmatrix}
  3 & 0 & 1 & 1 & 1 & 1 & 0 & 1 & 1  \\
  0 & 3 & 0 & 1 & 1 & 1 & 1 & -1 & 0 \\
  1 & 0 & 3 & 0 & 1 & 1 & 0 & 1 & -1 \\
  1 & 1 & 0 & 3 & 0 & 1 & 1 & 0 & 1 \\
  1 & 1 & 1 & 0 & 3 & 0 & -1 & 1 & 0 \\
  1 & 1 & 1 & 1 & 0 & 3 & 0 & -1 & 1 \\
  0 & 1 & 0 & 1 & -1 & 0 & 3 & 1 & 1 \\
  1 & -1 & 1 & 0 & 1 & -1 & 1 & 3 & 1 \\
  1 & 0 & -1 & 1 & 0 & 1 & 1 & 1 & 3 \\
\end{smallmatrix}%
\right].
$$
\bigbreak Then, by using Maple, we obtain that the solutions to the
system (\ref{u}) are
$$
\begin{array}{rcl}
v & = & (-\frac9{281},\frac {100}{281}-t_{1},\frac
{166}{281}-t_{1}-t_{2},\frac{28}{281},\frac 8{281}+t_{1},-\frac
{55}{281}+t_{1}+t_{2},t_{1},t_{2}, \\
 & & \frac {161}{281}-t_{1}-t_{2}),\qquad t_{1}, t_{2} \in \RR.
\end{array}
$$

Then, the system does not have a positive solution because the first
coordinate is always negative, and therefore by
Corollary \ref{ni} $\cg_{1,0}(8)$ is not Einstein nilradical.\\

\textbf{Case 2.} We consider the Lie algebra of rank $1$, of class $A_{4}$ and dimension $8$, $\dg_{1}(8)$, it
is nilpotent and has pre-Einstein derivation given by $\phi(e_{1})=e_{1}$, $\phi(e_{i})=(i+2)e_{i}$, $2\leq i
\leq 7$, with eigenvalues $\{1,4,5,6,7,8,9,10\}$).

The finite set $\mathbf{F}=\{\alpha_{ij}^{k}:c_{ij}^{k} \neq 0\}$
is:
$$
\begin{array}{ccc}
\{(-1,-1,1,0,0,0,0,0), & (-1,0,-1,1,0,0,0,0), & (-1,0,0,-1,1,0,0,0), \\
\ (-1,0,0,0,-1,1,0,0), & (-1,0,0,0,0,-1,1,0), & (-1,0,0,0,0,0,-1,1), \\
\ (0,-1,-1,0,0,0,1,0), & (0,-1,0,-1,0,0,0,1)\}, & \\
\end{array}
$$
\no and with respect to this enumeration, the matrix $U$ is

$$
U=\left[%
\begin{smallmatrix}
  3 & 0 & 1 & 1 & 1 & 1 & 0 & 1  \\
  0 & 3 & 0 & 1 & 1 & 1 & 1 & -1 \\
  1 & 0 & 3 & 0 & 1 & 1 & 0 & 1  \\
  1 & 1 & 0 & 3 & 0 & 1 & 0 & 0  \\
  1 & 1 & 1 & 0 & 3 & 0 & 1 & 0 \\
  1 & 1 & 1 & 1 & 0 & 3 & -1 & 1 \\
  0 & 1 & 0 & 0 & 1 & -1 & 3 & 1 \\
  1 & -1 & 1 & 0 & 0 & 1 & 1 & 3 \\
\end{smallmatrix}%
\right].
$$
\bigbreak Solutions to (\ref{u}) are
$$
\begin{array}{rcl}
v & = & (\frac {3}{62},-\frac
{7}{186}+t_{1},\frac{29}{186},\frac{20}{93},\frac{13}{93},\frac{32}{93}-t_1,\frac{77}{186}-t_1,t_1),\qquad
t_{1} \in \RR.
\end{array}
$$

Then by Corollary \ref{ni}, $\dg_{1}(8)$ is Einstein nilradical because the system has a positive solution
considering $ \frac{7}{186}< t_{1}< \frac
{64}{186}$.\\

\textbf{Case 3.} We consider the Lie algebras of rank $1$, class $A_{2}$ and dimension $8$, $\ggo_{\alpha}(8),
\alpha \in \RR$, which has pre-Einstein derivation with all the eigenvalues simple, ($\phi(e_{1})=e_{1},
\phi(e_{i})=ie_{i}, i \geq 2$, with eigenvalues $\{1,2,3,4,5,6,7,8 \}$).

If $\alpha \neq -2,-1,0$ then the finite set
$\mathbf{F}=\{\alpha_{ij}^{k}:c_{ij}^{k} \neq 0\}$ is given by:
$$
\begin{array}{ccc}
\{(-1,-1,1,0,0,0,0,0), & (-1,0,-1,1,0,0,0,0), & (-1,0,0,-1,1,0,0,0), \\
\ (-1,0,0,0,-1,1,0,0), & (-1,0,0,0,0,-1,1,0), & (-1,0,0,0,0,0,-1,1), \\
\ (0,0,-1,-1,0,0,1,0), & (0,0,-1,0,-1,0,0,1), & (0,-1,-1,0,1,0,0,0), \\
\ (0,-1,0,-1,0,1,0,0), & (0,-1,0,0,-1,0,1,0), & (0,-1,0,0,0,-1,0,1)\},\\
\end{array}
$$
Solutions of (\ref{u}) are
$$
\begin{array}{rcl}
v & = & (-\frac3{17}+t_{2}+t_{1}, -\frac
9{17}+t_{2}+t_{4}+t_{1}+t_{5}+t_{3}, -\frac
4{17}+t_{2}+t_{4}+t_{5}, \\
& & \frac {12}{17}-t_{2} -t_{4}-t_{1}-t_{3},\frac
{11}{17}-t_{2}-t_{4}-t_{1}-t_{5}, \frac 7{17}-t_{2}-t_{5}, t_{1},
t_{2}, \\
 & & \frac {14}{17}-t_{2}-t_{4}-t_{1}-t_{5} -t_{3}, t_{3},
t_{4}, t_{5}),\qquad t_{1},t_{2},t_{3},t_{4},t_{5} \in \RR.
\end{array}
$$

Therefore, if $\alpha \neq -2,-1,0, \ggo_{\alpha}(8)$ is Einstein nilradical by Corollary \ref{ni} because the
system has a positive solution, for example, taking $t_{1}= \frac {2}{17}, t_{2}= \frac {2}{17}, t_{3}= \frac
{1}{17}, t_{4}= \frac {5}{17}, t_{5}= \frac
{1}{17}$.\\

\begin{table}[h]
{\center
\begin{tabular}{|c|c|c|c|}
  \hline
Algebra & $\begin{array}{c}
             \mbox {Einstein} \\
             \mbox{nilradical}
             \end{array}$  &  $\begin{array}{c}
                               \mbox{No Nilradi-}\\
                               \mbox{cal Einstein}
                               \end{array}$ & eigenvalue type \\

\hline

$\mg_{0}(8)$ & \checkmark & & {\small $\begin{array}{c}
                                      1<26<27<28<29< \\
                                      30<31<32
                                       \end{array}$}\\

\hline

$\mg_{1}(8)$ & \checkmark & & {\small $\begin{array}{c}
                                      10<123<133<143<153< \\
                                      163<173<296
                                       \end{array}$}\\

\hline

$\mg_{2}(8)$ & & \checkmark & \\ 

\hline

$\ggo_{\alpha}(8)$, $\alpha \neq -2$ & \checkmark &  & {\small $1<2<3<4<5<6<7<8$}\\

\hline

$\ggo_{-2}(8)$ &  & \checkmark & \\

\hline

$\ag_{t}(8)$, $t \neq -1$ & \checkmark &  & {\small $1<3<4<5<6<7<8<9$}\\

\hline

$\ag_{-1}(8)$ &  & \checkmark & \\

\hline

$\cg_{1,0}(8)$ & & \checkmark & \\ 

\hline

$\dg_{1}(8)$ & \checkmark & & {\small $1<4<5<6<7<8<9<10$}\\

\hline

$\hg_{1}(8)$ & \checkmark & & {\small $1<5<6<7<8<9<10<11$}\\

\hline

$\bg(8)$ & \checkmark & &{\small $1<2<3<4<5<6<7<9$}\\

\hline

$\kg_{1}(8)$ & \checkmark &  & {\small $1<3<4<5<6<7<8<11$} \\

\hline

$\sg_1(8)$ & \checkmark & &{\small $1<4<5<6<7<8<9<13$}\\

\hline

\end{tabular}}
\vs \caption{Classification of filiform Einstein nilradicals of
dimension $8$}\label{tad}
\end{table}

If $\alpha=0$ then the finite set
$\mathbf{F}=\{\alpha_{ij}^{k}:c_{ij}^{k} \neq 0\}$ is
$$
\begin{array}{ccc}
\{(-1,-1,1,0,0,0,0,0), & (-1,0,-1,1,0,0,0,0), & (-1,0,0,-1,1,0,0,0), \\
\ (-1,0,0,0,-1,1,0,0), & (-1,0,0,0,0,-1,1,0), & (-1,0,0,0,0,0,-1,1), \\
\ (0,0,-1,-1,0,0,1,0), & (0,0,-1,0,-1,0,0,1), & (0,-1,-1,0,1,0,0,0), \\
\ (0,-1,0,-1,0,1,0,0), & (0,-1,0,0,-1,0,1,0)\}, &\\
\end{array}
$$

Observe that the only difference between the matrixes $U$ of the
cases $\alpha \neq -2,-1,0$ and $\alpha=0$ is that the second one is
obtained by erasing the last row and the last column of the first,
since the set $\mathbf{F}$ of the case $\alpha=0$ does not contain
the vector $\alpha_{26}^{8}$.

Solutions to $Uv=[1]$ are
$$
\begin{array}{rcl}
v & = & (\frac {11}{17}-t_{2}-t_{3}-t_{4},\frac 5{17}-t_{2}, \frac
{10}{17}-t_{2}-t_{3}-t_{1}, -\frac 2{17}+t_{2}, t_{2}+t_{3}-\frac
3{17}, \\
 & & -\frac 7{17}+t_{2}+t_{3}+t_{1}+t_{4}, t_{1}, \frac
{14}{17}-t_{2}-t_{3}-t_{1}-t_{4}, t_{2}, t_{3}, t_{4}),\qquad
t_{1},\\
 & & t_{2},t_{3},t_{4} \in \RR.
\end{array}
$$

This implies that $\ggo_{0}(8)$ is Einstein nilradical because by taking $t_{1}=\frac {2}{17},t_{2}=\frac
{4}{17},t_{3}=\frac {1}{17},t_{4}=\frac
{1}{17}$ we obtain a positive solution.\\

If $\alpha=-1$, $\mathbf{F}=\{\alpha_{ij}^{k}:c_{ij}^{k} \neq 0\}$
is given by
$$
\begin{array}{ccc}
\{(-1,-1,1,0,0,0,0,0), & (-1,0,-1,1,0,0,0,0), & (-1,0,0,-1,1,0,0,0), \\
\ (-1,0,0,0,-1,1,0,0), & (-1,0,0,0,0,-1,1,0), & (-1,0,0,0,0,0,-1,1), \\
\ (0,0,-1,-1,0,0,1,0), & (0,0,-1,0,-1,0,0,1), & (0,-1,-1,0,1,0,0,0), \\
\ (0,-1,0,-1,0,1,0,0), & (0,-1,0,0,0,-1,0,1)\}, &\\
\end{array}
$$

Solution of (\ref{u}) are
$$
\begin{array}{rcl}
v & = & (\frac 8{17}-t_1-t_4, \frac 2{17}-t_1+t_3, -\frac
4{17}+t_4+t_2, \frac 1{17}+t_1-t_3+t_4, t_1,\\
& & \frac 7{17}-t_4-t_2, \frac {11}{17}-t_1-t_4-t_2, t_2, \frac
3{17}+t_1-t_3, t_3, t_4),\\
& & t_{1},t_{2},t_{3},t_{4} \in
 \RR.
\end{array}
$$
Therefore, $\ggo_{-1}(8)$ is Einstein nilradical taking, for
example, $t_{1}= \frac 1{17}, t_{2}= \frac 4{17}, t_{3}= \frac
1{17}, t_{4}= \frac 1{17}$.\\

If $\alpha=-2$, the finite set $\mathbf{F}$ is
$$
\begin{array}{ccc}
\{(-1,-1,1,0,0,0,0,0), & (-1,0,-1,1,0,0,0,0), & (-1,0,0,-1,1,0,0,0), \\
\ (-1,0,0,0,-1,1,0,0), & (-1,0,0,0,0,-1,1,0), & (-1,0,0,0,0,0,-1,1), \\
\ (0,0,-1,-1,0,0,1,0), & (0,0,-1,0,-1,0,0,1), & (0,-1,0,0,-1,0,1,0), \\
\ (0,-1,0,0,0,-1,0,1)\}, & &\\
\end{array}
$$

Solutions of system (\ref{u}) are
$$
\begin{array}{rcl}
v & = & (-\frac 3{17}+t_{2}+t_{1}, \frac 5{17}, \frac
{10}{17}-t_{1}, -\frac 2{17}+t_{3}, -\frac 3{17}, \frac
7{17}-t_{2}-t_{3}, t_{1},
t_{2}, \\
 & & \frac {14}{17}-t_{2}-t_{1}-t_{3},t_{3}),\qquad
t_{1},t_{2},t_{3} \in \RR.
\end{array}
$$

So $\ggo_{-2}(8)$ is not Einstein nilradical because the system does
not has a positive solution since the fifth coordinates of
$v$ is always negative.\\

\begin{theorem}
Let $\ngo$ be a filiform Lie algebra of dimension $8$. Then $\ngo$
is an Einstein nilradical if and only if $\ngo$ is not any of the
following algebras: $\mg_{2}(8)$, $\ggo_{-2}(8)$, $\ag_{-1}(8)$,
$\cg_{1,0}(8)$ (see Table \ref{tad}).
\end{theorem}

\end{document}